\documentclass{article}
\usepackage[utf8]{inputenc}
\usepackage{subfigure}
\usepackage[round]{natbib}
\usepackage{amsmath}
\usepackage{amsthm}
\usepackage{amssymb}
\usepackage{hyperref}
\newtheorem{prop}{Proposition}[section]

\newtheorem{thm}[prop]{Theorem}

\theoremstyle{definition}
\newtheorem{eg}[prop]{Example}
\def \ds {\displaystyle} \def \tss {\textsuperscript}             \def \Fix {\text{Fix}} \def \Red {\text{Reduce}} \def \Rem {\text{Remove}}

\author{Yonah Biers-Ariel}
\title{Counting Words Avoiding a Short Increasing Pattern and the Pattern 1k\dots2}

\begin{document}\sloppy
\maketitle

\begin{abstract}
We find finite-state recurrences to enumerate the words on the alphabet $[n]^r$ which avoid the patterns 123 and $1k(k-1)\dots2$, and, separately, the words which avoid the patterns 1234 and $1k(k-1)\dots2$.
\end{abstract}

\section{Introduction}
A word $W = w_1w_2 \dots w_n$ on an ordered alphabet \emph{contains} the pattern $p_1p_2\dots p_k$ if there exists a (strictly) increasing sequence $i_1,i_2,\dots, i_k$ such that $w_{i_r} < w_{i_s}$ if and only if $p_r < p_s$ and $w_{i_r} > w_{i_s}$ if and only if $p_r > p_s$. If both $W$ and $p_1p_2 \dots p_k$ are permutations, then $w_{i_r} < w_{i_s}$ if and only if $p_r < p_s$ is an equivalent and more common definition. If $W$ does not contain $p_1p_2 \dots p_k$, then $W$ \emph{avoids} it. The study of pattern-avoiding permutations began with Donald Knuth in \textit{The Art of Computer Programming}, and has become an active area of combinatorial research. See \cite{Vatter} for an in-depth survey of the major results in this field.

The study of pattern-avoiding words other than permutations is comparatively recent, being inaugurated in \cite{Regev} and greatly expanded in \cite{Burstein}. Much is known about avoidance properties for specific patterns and families of patterns; for instance, Burstein counted the number of length-$n$ words with letters in $[k]=\{1,2,\dots,k\}$ which avoid all patterns in $S$ for all $S \subseteq S_3$. Meanwhile, \cite{Mansour} found generating functions for the number of such words which avoid both 132 and one of a large family of other patterns including $12\dots l$ and $l12\dots(l-1)$. 

Other authors have looked at words with letters in $[n]$ where each letter must appear exactly $r$ times (we will say that these are the words on $[n]^r$). These words are a direct generalization of permutations, which are given by the $r=1$ case. In \cite{Shar}, the authors created an algorithm to find the ordinary generating functions enumerating words on $[n]^r$ which avoid 123, while \cite{Zeil1d} found that the generating functions enumerating words on $[n]^r$ avoiding $12\dots l$ are D-finite. 

We study this second type of word. Our contribution is to find finite, linear recurrences for the numbers of words on $[n]^r$ that avoid 123 and $1k(k-1)\dots2$ as well as the ones that avoid 1234 and $1k(k-1)\dots2$. It is well known (see \cite{Cfinite} for instance) that this fact implies that these quantities have rational generating functions, and, moreover, gives a way to compute them (in principle if not always in practice - see Section \ref{compute}). While generating functions were previously found in \cite{Kratten} for the permutations avoiding 123 and $1k(k-1)\dots2$, this is the first time that such a result has been extended to these more general words. In the 1234 and $1k(k-1)\dots2$ case, the result was, to the best of our knowledge, previously not even known for permutations with $k$ as small as 5.

\section{Words Avoiding 123}\label{123}
We begin this section with an algorithm for counting 123 avoiding permutations from \cite{DrZ}. For $L = [l_1,l_2,\dots l_n]$, let $A(L)$ be the number of words containing $l_i$ copies of $i$ for $1 \le i \le n$ which avoid 123. The following result allows us to quickly compute $A(L)$.
\begin{thm} The following recurrence holds: 
\begin{displaymath}
A(L) = \sum_{i=1}^n A([l_1,l_2,\dots,l_{i-1}, l_i-1, l_{i+1}+l_{i+2}+\dots+l_n]).
\end{displaymath}
\end{thm}
\begin{proof} Let $A_i(L)$ be the number of words with letter counts $l_1,l_2,\dots,l_n$ which avoid 123 and begin with the letter $i$. We will biject the words counted by $A_i(L)$ with those counted by $A([l_1,l_2,\dots,l_{i-1}, l_i-1, l_{i+1}+l_{i+2}+\dots+l_n])$. Let $W=iw_2,\dots,w_t$ have letter counts in $L$ and let $f(W)$ be given by removing the initial $i$ from $W$ and then replacing all letters greater than $i$ with $i+1$. Also, for some word $V = v_1v_2,\dots,v_{t-1}$ with letter counts in $[l_1,l_2,\dots,l_{i-1},l_i-1,l_{i+1}+\dots+l_n]$, let $f^{-1}$ be given by replacing the sequence of $i+1$'s with $l_{n} \text{ }n$'s, $l_{n-1} \text{ }n-1$'s, and so on in that order, and then prepending $i$ to this word.

We claim that $f^{-1}$ is the inverse of $f$. To find $f(f^{-1}(V))$, we would replace the sequence of $i+1$s with $l_{n} \text{ }n$'s, $l_{n-1} \text{ }n-1$'s, and so on, and then prepend an $i$, before removing that $i$ and replacing all those letters larger than $i$ with $i+1$ again, giving us back $V$. To find $f^{-1}f(W)$, we would replace all the letters larger than $i$ with $i+1$ and remove the initial $i$, before replacing that $i$ and putting back all the letters larger than $i$ (note that they had to be in descending order to begin with or else $W$ would contain a 123 pattern). Thus, $A_i(L)=A([l_1,l_2,\dots,l_{i-1}, l_i-1, l_{i+1}+l_{i+2}+\dots+l_n])$ and summing over all $i$ gives the promised equality.
\end{proof}

This technique can be extended to many more avoidance classes. In this section we use it to count words avoiding both $123$ and $1k(k-1)\dots2$ simultaneously. We first fix $k \ge 3$, choose integers $n$ and $r$, and consider the number of words on the alphabet $[n]^r$ which avoid both $123$ and $1k(k-1)\dots2$. This time, however, we will need to keep track of more information than just the letter counts. To that end, we consider the set of words $A(r,a,b,L)$ where $r,a,$ and $b$ are integers, and $L=[l_1,l_2,\dots,l_t]$. This is the number of words with $r$ copies of the letters $1, \dots, a$, $b$ copies of the letter $a+1$, and $l_i$ copies of the letter $a+1+i$ which not only avoid both 123 and $1k(k-1)\dots2$, but would still avoid both those patterns if $a+1$ were prepended to the word. Note that this condition implies $t \le k-2$ because any sequence of $k-1$ distinct letters will either contain an increasing subsequence of length 2 (and hence create a 123 pattern) or will be entirely decreasing (and hence create a $1k(k-1)\dots2$ pattern).

Before we state the next theorem, we describe in more human-friendly language the algorithm that it suggests. Suppose that we are building a word $W$. Up to this point, the smallest letter which has been used is $a+1$, and $r-b$ copies of it have been used. We have a list $L=[l_1,l_2,\dots l_t]$ indicating how many copies of each letter greater than $a+1$ remain to be added, and we note that all these letters must be added in reverse order. To complete $W$, we need to add $r$ copies each of $1,2,\dots,a$, $b$ copies of $a+1$, and $l_i$ copies of $a+1+i$ for all $1 \le i \le t$. Examine $W$'s next letter $w_1$; considering only the requirement that $w_1$ be succeeded by at most $k-2$ distinct letters larger than it, we find that $w_1$ can be any element of $\{a+2, a+3,\dots,a+1+t\}$ or else it can be an element of $\{a-(k-2)+t+1, a-(k-2)+t+2, \dots, a+1\}$. But, we also need to consider the requirement that prepending $a+1$ to the new word will not create a 123 pattern. Therefore, $w_1 \in \{a-(k-2)+t+1, a-(k-2)+t+2, \dots, a+1, a+1+t\}$. If $w_1 = a+1+t$ or $a+1$, then removing it gives a word counted by $A(r,a,b,L')$ or $A(r,a,b-1,L)$ respectively where $L' = [l_1,\dots,l_{t-1},l_t -1]$. Otherwise, we need to add all the letters larger than $w_1$ to $L$ in order to ensure that future letters don't create 123 patterns.

Since we want $L$ to contain only letter counts for letters which will be added to $W$, i.e. we don't want it to contain 0, define the operator $R$ which removes all the zeroes from the list $L$.

\begin{thm} If $b \ge 1$, then
\begin{align*} A(r,a,b,L) &= \sum_{i = a-(k-2)+t+1}^{a} A(r,i-1,r-1,[\underbrace{r,r,\dots,r,}_{a-i \text{ copies}} b, l_1,\dots,l_t])
\\
&+ A(r,a,b-1,L) + A(r,a,b,R([l_1,l_2,\dots, l_t -1])).
\end{align*}
If $b =0$, then
\begin{align*} A(r,a,b,L) &= \sum_{i = a-(k-2)+t+1}^{a} A(r,i-1,r-1,[\underbrace{r,r,\dots,r,}_{a-i \text{ copies}} l_1,\dots,l_t])
\\
&+A(r,a,b,R([l_1,l_2,\dots, l_t -1])).
\end{align*}
\end{thm}

\begin{proof} As noted in the previous paragraph, $w_1 \in \{a - (k-2) +t +1, a-(k-2)+t+2,\dots,a+1,a+1+t\}$; each member of this set corresponds to a term of the summation. Fix $i$ with $a-(k-2)+t+1 \le i \le a$, and consider those words $W$ with $w_1 =i$. Suppose we remove $w_1$ from one of these words to form a word $W'$. We are left with $r$ copies of the letters 1 through $i-1$, $r-1$ copies of $i$ (because $i \le a$ there were $r$ copies of it including $w_1$), and $l'_j$ copies of $(i-1)+1+j$ where $L'=[l'_1,\dots,l'_u]=[\underbrace{r,r,\dots,r}_{a-i \text{ copies}},b,l_1,\dots,l_t]$. We are left with a word which avoids 123 and $1k(k-1)\dots2$, and, moreover, avoids 123 even when $i$ is prepended. Furthermore, prepending an $i$ to any word fitting this description gives a word counted by $A(r,a,b,L)$, and so the number of words counted by $A(r,a,b,L)$ which begin with $i$ is $A(r,i-1,r-1,[\underbrace{r,r,\dots,r,}_{a-i \text{ copies}} b, l_1,\dots,l_t])$ for all $a-(k-1)+t+1 \le i \le a$.

This leaves two other possibilities for $w_1$: $a+1$ and $a+t+1$. If $w_1 = a+1$, then the only difference between the letter counts of $W$ and $W'$ is that $W$ has $b$ copies of $a+1$ and $W'$ has only $b-1$. In terms of avoidance, both $W$ and $W'$ avoid 123 and $1k(k-1)\dots2$ even with $a+1$ prepended. Thus, the number of $W$ with $w_1=a+1$ is $A(r,a,b-1,L)$ as long as $b \ge 1$, and, 0 if $b=0$.

Similarly, if $w_1=a+t+1$, then the only difference between the letter counts of $W$ and $W'$ is that $W$ has $l_t$ copies of $a+t+1$ and $W'$ has $a+t$. Just as in the previous case, the avoidance properties are identical and so the number of $W$ with $w_1=a+t+1$ is $A(r,a,b,R([l_1,l_2,\dots,l_t-1]))$ where we needed to remove $l_t-1$ if it is zero so that we know that the next letter is allowed to be $l_{t-1}$.

Summing over all possible $w_1$ now gives the promised result.

\end{proof}

\section{Words Avoiding 1234}\label{1234}
Just as we can find recurrences, and therefore generating functions, for words avoiding 123 and $1k(k-1)\dots2$, we can (in principle at least) find a similar system of recurrences and generating functions for words on $[n]^r$ avoiding 1234 and $1k(k-1)\dots2$. The idea is to construct a word $W$ one letter at a time, and with each letter see if we have made a forbidden pattern. Unfortunately, doing this naively would require keeping track of all previous letters in $W$, denying us a finite recurrence. By only paying attention to the letters that could actually contribute to a forbidden pattern, though, we find that we actually only need to retain a bounded quantity of information regarding $W$.

\subsection{The Existence of a Finite Recurrence}
In order to discuss the structure of words avoiding 1234 and $1k(k-1)\dots2$, we recall one common definition and introduce some new ones. A \emph{left-to-right minimum} (LTR min) is a letter of a word which is (strictly) smaller than all the letters which precede it. To an LTR min, we associate an \emph{activated sequence} which consists of all the letters following the LTR min which are (again strictly) larger. Notice that, since 1234 is forbidden, anytime a letter $w$ is preceded by some smaller letter, all the letters larger than and following $w$ must occur in reverse order. We call these letters \emph{fixed}. If all the letters greater than an LTR min are fixed, then it is either guaranteed or impossible that the LTR min and its activated sequence form a $1k(k-1)..2$ pattern; in this case we say that the activated sequence has been \emph{deactivated} and we no longer consider it an activated sequence. If an LTR min with an empty activated sequence is followed by another LTR min (or another copy of itself), then any forbidden pattern using the first LTR min could also be made using the second LTR min; we say that the first LTR min is \emph{superceded} and no longer consider it an LTR min.

With these definitions, we are nearly ready to state the actual set we will be recursively enumerating. Let $r,k,$ and $a$ be integers, let $\mathcal{S}=[S_1=[s_{1,1},\dots,s_{1,q_1}],\dots, S_u=[s_{u,1},\dots,s_{q_u}]]$ be a list of lists whose elements are in $[t]$, let $M=[m_1,\dots,m_u]$ be a list with elements in $[t]$, and let $L=[l_1,\dots,l_t]$ be a list with elements in $\{0\} \cup [r]$. Suppose we are building a word, and so far the letters $1,\dots,a$ have never been used, while the letters greater than $a+t+1$ have been entirely used up and, moreover, are not LTR mins or in any activated sequence. Suppose this word has LTR mins $m_1+a,\dots,m_s+a$ with corresponding activated sequences $[s_{1,1}+a,\dots,s_{1,q_1}+a],\dots,[s_{u,1}+a,\dots,s_{u,q_u}+a]$ (excluding LTR mins which have been superceded or whose sequences have been deactivated). Finally, assume that the word so far avoids 1234 and $1k(k-1)\dots2$, and that $L_i$ copies of $a+i$ remain to be placed for all $1 \le i \le t$ (recall that $r$ copies of $1,\dots,a$ and 0 copies of $a+t+1,a+t+2,\dots$ remain to be placed). Then, the number of ways to completing the word is defined to be $A(r,k,a,M,\mathcal{S},L)$.

Our plan is to show that (i) $A$ is well defined, (ii) all arguments of $A$ besides $a$ take on finitely many values, and (iii) $A$ satisfies a recurrence in which a particular non-negative function of its arguments is always reduced (until the base case). Before carrying out this plan, though, we provide an example to make sure our definitions are clear.

\begin{eg}\label{ex}
Suppose we are building a word on the alphabet $[9]^2$ to avoid 1234 and 15432, and so far have 69945. The LTR mins are 6 and 4 with corresponding activated strings 99 and 5. However, 99 has been deactivated because all the letters greater than its LTR min are fixed and must occur in decreasing order. Thus, the number of ways to complete this word is given by $A(2,5,3,[1],[[2]],[1,1,1,2,2]).$ 
\end{eg}

Notice that the number of ways is also given by $A(2,5,2,[2],[[3]],[2,1,1,1,2,2])$. While this is not a problem in principle, it would be nice to have a canonical way of expressing this quantity, and so we will eventually insist that $L$ have a particular length given by a function of $r$ and $k$.

\begin{thm}\label{well-defined}
$A$ is well defined.
\end{thm}
\begin{proof}
Suppose that $W_1$ and $W_2$ are two partial words that give the same arguments to $A$.

Given the LTR mins of a partial word, it is easy to see in which order they occurred. It is similarly easy to see in which order the elements of activated sequences occurred, since they are all listed in order in the activated sequence corresponding to the first LTR min (any element of some other activated sequence lower than or equal to the first LTR min would fix that LTR min's activated sequence and thus deactivate it). Finally, we can see how the sequence of LTR mins and the sequence of other elements are interweaved by noting that a non-LTR min occurs after an LTR min if and only if it appears in that min's activated sequence. Therefore, the subwords formed by the LTR mins and activated strings of $W_1$ and $W_2$ are identical.

Suppose $A$ is not well-defined; then there is some string which can be added to (without loss of generality) $W_1$ without creating a forbidden pattern, but which does create a forbidden pattern when added to $W_2$. By the argument of the previous paragraph, there is an element of $W_2$ which is neither an LTR min nor part of an activated sequence, but which does participate in this pattern. 

But, this is not possible. Every element in $W_2$ is an LTR min, part of an activated sequence, fixed, or a copy of the LTR min immediately preceding it. We have assumed that the first two cases do not hold. The third case similarly cannot hold because when an element is fixed, so are all the elements larger than its LTR min, which is to say all the elements which could conceivably be part of a forbidden pattern with it. Therefore, every fixed element either must participate in a forbidden pattern or it cannot possibly do so. Finally the fourth case cannot hold because any forbidden pattern involving a copy of the immediately preceding LTR min could also be formed with that LTR min. Thus we have a contradiction.
\end{proof}

Next, we want to establish bounds on $s$ and $t$ as well on the number of elements in any $S_i$. These bounds should depend only on $r$ and $k$.

\begin{thm}
Bounds for $t$, $u$, and all $|S_i|$ are as follows: $t \le 6(k-2)+2$, $u \le 2r(k-2)+1$, and $|S_i| \le 2r(k-2)$ for all $1 \le i \le u$.
\end{thm}

\begin{proof}
Recall that the Erd\H os-Szekeres Theorem states that any sequence of distinct real numbers of length $(p-1)(q-1)+1$ must contain either a length$-p$ increasing sequence or a length$-q$ decreasing sequence (see \cite{ESz}). Since every activated sequence must avoid both 123 and $(k-1)(k-2)\dots1$, it follows that no activated sequence can have more than $2(k-2)$ distinct letters. Since there are at most $r$ copies of any single letter, the longest an activated sequence could possibly be is $2r(k-2)$. Each activated sequence must either have some element that the next one lacks or correspond to the most recent LTR min (or else its LTR min would be superseded), and in the proof of Theorem \ref{well-defined} we showed that the first activated sequence must contain all the elements of every other activated sequence. Therefore, there can be at most $2r(k-2)+1$ activated sequences, and we have successfully bounded both $u$ and the size of any $S_i$.  

To find a bound on $t$, note that as soon as we have used all $r$ copies of a letter and none of those copies remain as either LTR mins or in activated sequences, we can ignore that letter entirely, secure in the knowledge that if it is not already part of a forbidden pattern, it never will be. Thus, we only need to keep track of letters which are LTR mins, are in activated sequences, or are among the largest $2(k-2)+1$ letters still available to be used. By the reasoning of the previous paragraph, at most $2(k-2)+1$ distinct letters are LTR mins, at most $2(k-2)$ are in activated sequences, and so we need to keep track of $6(k-2)+2$ letters all together; all other letters either have never been used and so have $r$ copies remaining or else have had all copies used and can no longer participate in forbidden patterns.
\end{proof}

After Example \ref{ex}, we commented that there may be several ways to describe a given partial word using $a,L,M,$ and $\mathcal{S}$. To allow for unique descriptions, we adopt the convention that $|L| = t = 6(k-2)+2$.

\subsection{Finding the Recurrence}

At this point, we have finished parts (i) and (ii) of our program; all that's left to show is that $A$ can be computed using a recurrence. To this end, we introduce three new functions. $\Fix(m,S,L)$ returns $S$ with all available letters (with counts determined by $L$) which exceed $m$ appended in decreasing order, $\Red(L,i)$ returns $L$ with the $i\tss{th}$ element decreased by one, and $\Rem(r,M,\mathcal{S},L,i)$ returns the tuple $M,\mathcal{S},L$ with the following changes: all elements of $M$ and all elements of every $S$ in $\mathcal{S}$ that are below $i$ are incremented by one, the $i\tss{th}$ element of $L$ is deleted and $r$ is prepended to $L$. 

\begin{eg}\label{fix}
$\Fix(1,[2,3],[1,1,1,2,2]) = [2,3,5,5,4,4,3,2].$
\end{eg}

\begin{eg}\label{reduce}
$\Red([1,1,1,2,2],2) = [1,0,1,2,2].$
\end{eg}

\begin{eg}\label{remove}
$\Rem(2,[1],[[2,4]],[1,1,0,1,2],3) = ([2],[[3,4]],[2,1,1,1,2]).$
\end{eg}

For fixed $r$ and $k$, the base cases are $A(r,k,a,M,\mathcal{S},L)$ for all $M,\mathcal{S},  L$ and $0\le a \le t$. Otherwise, there are no more than $2(k-2)+1$ possibilities for the next letter $i$ (corresponding to the largest $2(k-2)+1$ nonzero entries of $L$) which we divide into $u+1$ cases: when $i \le m_u$, when $m_h < i \le m_{h-1}$ for $2 \le h \le u-1$, and when $m_1 < i$. Suppose that $i \le m_u$; then we calculate the number of ways to complete the word after adding an $i$ as follows.

Suppose that $S_u = []$, then adding an element less than or equal to the current LTR min will supercede that LTR min. Symbolically, we have $\mathcal{S}' = [S_1,\dots,S_{u-1},[]]$, $M' = [m_1,\dots,m_{u-1},i]$, and $L' = \Red(L,i)$. If $S_u \neq []$, then we are adding a new LTR min while leaving all existing ones in place. This gives arguments $\mathcal{S}' = [S_1,\dots,S_{u},[]]$, $M'=[m_1,\dots,m_u,i]$, and $L' = \Red(L,i)$. Now, suppose that $J = \{j_1,\dots,j_w\}$ is a set of all the the integers $j \in [t]$ such that $L'_j = 0$, and $j$ fails to appear in $M'$ or in any $S \in \mathcal{S}'$. For all $j \in J$ from smallest to largest, update $\mathcal{S}',M',$ and $L'$ by setting $M',\mathcal{S}',L' = \Rem(r,M',\mathcal{S}',L',j)$. We finally have that the number of ways to complete the word after adding an $i$ is $A(r,k,a-|J|,M',\mathcal{S}',L')$.

Alternatively, we may add $i$ such that $i > m_h$ with $h$ chosen as small as possible. Either this $i$ is no larger than the smallest element of $S_1$, or else it is the largest letter that still remains to be added (otherwise a 1234 pattern is inevitable once the largest remaining letter is added). For all $m_j \ge i$, check to see if $\Fix(m_j,S_j,L)$ contains a $(k-1)(k-2)\dots1$ pattern. If so, this choice of $i$ contributes nothing to $A(r,k,a,M,\mathcal{S},L)$. If this is not true for any $j$, then we can add $i$ to our word, but doing so deactivates $S_1,S_2,\dots,S_{h-1}$, and so we forget about those activated sequences and their LTR mins. Thus, we take $\mathcal{S}' = [S_h,\dots,S_u]$, $M'  = [M_h,\dots,M_u]$, and $L' = \Red(L,i)$. As before, suppose that $J = \{j_1,\dots,j_w\}$ is a set of all the the integers $j \in [t]$ such that $L'_j = 0$, and $j$ fails to appear in $M'$ or in any $S \in \mathcal{S}'$. For all $j \in J$ from smallest to largest, update $M',\mathcal{S}',$ and $L'$ by setting $M',\mathcal{S}',L' = \Rem(r,M',\mathcal{S}',L',j)$. Again we have that the number of ways to complete the word after adding an $i$ is $A(r,k,a-|J|,\mathcal{S}',M',L')$.

We have expressed $A(r,k,a,M,\mathcal{S},L)$ as a sum of other terms. Notice that in each of these other terms, the number of letters left to be added (given by $r\cdot a +\sum_{i=1}^t L_i$) decreases by 1; eventually it will decrease below $r\cdot t$ and a base case will apply.

While all the base cases could in principle be computed individually, this would probably be a long and unpleasant task. Fortunately, our recurrence can be easily tweaked to calculate base cases. To do so, simply run the recurrence as given, but anytime $A$ would be called with a negative $a$, replace the first $|a|$ nonzero entries of $L$ with 0 and change $a$ to 0. As it turns out, the only base case that we really need is $A(r,k,0,M,\mathcal{S},[0,0,\dots,0])=1$. A full implementation of this recurrence is available in an accompanying Maple package -- see Section \ref{maple}.

\subsection{From Recurrences to Generating Functions}\label{rigor}
This subsection contains an algorithm for turning the recurrences found in this section and Section \ref{123} into generating functions. Readers interested in a more complete exposition should consult Chapter 4 in \cite{Kauers}. 

Suppose the different terms in our system of recurrences are given by $A_1(n), A_2(n),\dots A_m(n)$. While we could choose $n$ like before and let it be the number of distinct unused letters whose counts do not appear in $L$, the rest of this process will be easier if each $A_i(n)$ depends only on $A_j(n-1)$. To make this happen, we interpret $n$ as the total number of unused letters. For example, we might fix $r=2,k=5$ (note that $t=|L|$ is then chosen to be 20), and let $A_1(n) = A(2,5,n,[],[],[2,2,\dots,2]).$ If we let $A_2(n) = A(2,5,n,[1],[[]],[1,2,\dots,2])$, $A_3(n) = A(2,5,n,[2],[[]],[2,1,2,\dots,2])$ and so on until $A_{21}(n) = A(2,5,n,[20],[[]],[2,\dots,2,1])$, we find the recurrence relation $A_1(n) = \sum_{i=2}^{21} A_i(n-1)$. 

Let $M$ be the matrix whose $i,j$ entry is the coefficient of $A_j(n-1)$ in the recurrence for $A_i(n)$, and let $f_i(x)$ be the generating function $\sum_{n=0}^\infty A_i(n) x^n$. It follows that $f_i(x)$ is a rational function with denominator $\text{det}(xI-M)$ for all $i$. The numerator of each generating function has degree less than the number of rows of $M$, and the coefficients of each one can be determined using the system of recurrences' initial conditions.

Since we are treating $n$ as the total number of letters in a word, we must make the substitution $x^r \mapsto x$ in order to obtain the generating function for the number of words on the alphabet $[n]^r$ avoiding the two patterns.

\section{Computational Results}\label{compute}
The first algorithm presented in this paper, the one which enumerates words avoiding 123 and $1k(k-1)\dots2$, runs very quickly. With $r = 2$, we are able to get generating functions for $k$ as large as 8 (and we could go even further if we chose to). With $r = 3$, we are able to get generating functions for $k$ as large as 7.

Unfortunately, the algorithm in Section \ref{1234} is much slower. In the simplest open case of $r=1, k=5$, we are able to conjecture the generating function to be $\ds \frac{-2x^3+7x^2-6x+1}{2x^4-11x^3+17x^2-8x+1}$, but rigorously deriving it seems to be out of the question without carefully pruning the recurrence. We can also use the recurrence to just generate terms without worrying about finding generating functions. With $r =1$, i.e. in the permutation case, we find ten terms apiece in the enumeration sequences for $k=3,4 \dots, 10$, and could easily get more terms; in fact in the particular case of $k=5$ we found 16 in 20 minutes.

All these results can be found in the output files on this paper's webpage (see Section \ref{maple}).

\section{Future Work}
The driving force behind the argument in this paper is the Erdo\H s-Szekeres theorem; it ensures that we only have finitely many possible letters to add to a word at any point in time. For any pair of patterns which are not of the form $12\dots l, 1k(k-1)\dots 2$, this theorem will not apply, and so it is difficult to see how strategies like those in this paper could work.

It does seems reasonable to hope that they would work for other patterns of the form $12\dots l, 1k(k-1)\dots 2$. The only problem with applying them to the pair 12345, $1k(k-1)\dots 2$ is that we lose the fact that any element greater than an LTR min immediately fixes all elements above it. As a result, it is possible to have multiple activated strings, neither of which is a subset of the other. However, we are hopeful that some clever idea can get around this obstacle.

\section{Maple Implementation}\label{maple}
This paper is accompanied by three Maple packages available from the paper's website: \url{http://sites.math.rutgers.edu/~yb165/SchemesForWords/SchemesForWords.html}. The packages are {\tt 123Avoid.txt} which implements the recurrence described in Section \ref{123}, {\tt 123Recurrences.txt} which uses this recurrence to rigorously find the generating functions enumerating the words on $[n]^r$ avoiding 123 and $1k(k-1)\dots2$, and {\tt 1234Avoid.txt} which implements the recurrence described in Section \ref{1234}. It also uses Doron Zeilberger's package Cfinite to automatically conjecture generating functions for the sequences of numbers of words on $[n]^r$ avoiding 1234 and $1k(k-1)\dots2$.

After loading any of these packages, type {\tt Help();} to see a list of available functions. You can get more details about any function by calling Help again with the function's name as an argument. This will also give an example of the function's usage.

\section{Acknowledgements}
The author is grateful to his advisor, Doron Zeilberger, for suggesting the problem to him and for his frequent suggestions and improvements.

\nocite{*}
\bibliography{CountingWords}
\bibliographystyle{abbrvnat}

\end{document}